\documentclass[11pt,oneside,notitlepage]{article}

\usepackage[english]{babel}
\usepackage{amsmath}
\usepackage{amssymb}
\usepackage{amsthm}
\usepackage{mathtools}
\usepackage{xspace}

\usepackage{fullpage}

\theoremstyle{plain}
\newtheorem{theorem}{Theorem}[section]

\newtheorem{lemma}[theorem]{Lemma}
\newtheorem{corollary}[theorem]{Corollary}

\newtheorem{definition}[theorem]{Definition}
\theoremstyle{definition}
\newtheorem{example}[theorem]{Example}

\makeatletter
\let\@fnsymbol\@arabic
\makeatother

\newcommand{\nop}[1]{}

\newcommand{\xbar}{{\overline x}}

\newcommand{\bbr}{\ensuremath{\mathbb{R}}}

\newcommand{\n} [1]{\ensuremath{\left[#1\right]^-}}
\newcommand{\ve}[1]{{ \bf #1}}

\begin{document}

\title{
  Quadratization of Symmetric Pseudo-Boolean Functions
}

\author{
  Martin Anthony\thanks{Department of Mathematics, London School of Economics, UK. \emph{E-mail:} m.anthony@lse.ac.uk}
  \qquad
  Endre Boros\thanks{MSIS and RUTCOR, Rutgers University, NJ, USA. \emph{E-mail:} endre.boros@rutgers.edu}
  \qquad
  Yves Crama\thanks{QuantOM, HEC Management School, University of Liege, Belgium. \emph{E-mail:} yves.crama@ulg.ac.be}
  \qquad
  Aritanan Gruber\thanks{RUTCOR, Rutgers University, NJ, USA. \emph{E-mail:} aritanan.gruber@rutgers.edu}
}

\maketitle

\begin{abstract}
  A {\em pseudo-Boolean function} is a real-valued function $f(x)=f(x_1,x_2,\ldots,x_n)$ of $n$ binary variables; 
  that is, a mapping from $\{0,1\}^n$ to ${\bbr}$. For a pseudo-Boolean function $f(x)$ on $\{0,1\}^n$, we say that 
  $g(x,y)$ is a \emph{quadratization} of $f$ if $g(x,y)$ is a quadratic polynomial depending on $x$ and on $m$ 
  \emph{auxiliary} binary variables $y_1,y_2,\ldots,y_m$ such that $f(x)= \min \{ g(x,y) : y \in \{0,1\}^m \}$ for 
  all $x \in \{0,1\}^n$. By means of quadratizations, minimization of $f$ is reduced to minimization (over its 
  extended set of variables) of the quadratic function $g(x,y)$. This is of some practical interest because 
  minimization of quadratic functions has been thoroughly studied for the last few decades, and much progress has 
  been made in solving such problems exactly or heuristically. A related paper~\cite{ABCG} initiated a systematic 
  study of the minimum number of auxiliary $y$-variables required in a quadratization of an arbitrary function $f$ 
  (a natural question, since the complexity of  minimizing the quadratic function $g(x,y)$ depends, among other 
  factors, on the number of binary variables). In this paper, we determine more precisely the number of auxiliary 
  variables required by quadratizations of \emph{symmetric} pseudo-Boolean functions $f(x)$, those functions whose 
  value depends only on the Hamming weight of the input $x$ (the number of variables equal to $1$).
  \\[4mm]
  \\[4mm]
\end{abstract}
 
\section{Quadratizations of pseudo-Boolean functions}
A {\em pseudo-Boolean function}  is a real-valued function $f(x)=f(x_1,x_2,\ldots,x_n)$ of $n$ binary 
variables, a mapping from $\{0,1\}^n$ to ${\bbr}$. It is well-known that every pseudo-Boolean function can 
be uniquely represented as a multilinear polynomial in its variables. \emph{Nonlinear binary optimization 
problems}, or \emph{pseudo-Boolean optimization} (PBO) \emph{problems}, of the form
\begin{equation*}\label{eq:PBO}  
  \min \bigl\{ f(x) : x \in \{0,1\}^n \bigr\},
\end{equation*}
where $f(x)$ is a pseudo-Boolean function, have attracted the attention of numerous researchers, and 
they are notoriously difficult, as they naturally encompass a broad variety of models such as maximum 
satisfiability, maximum cut, graph coloring, simple plant location, and so on; see, e.g.,~\cite{CH2011}. 
In recent years, several authors have revisited an approach initially proposed by Rosenberg~\cite{Rosenberg75}. 
This involves reducing PBO to its quadratic case (QPBO) by relying on the following concept.

\begin{definition}\label{def:quad} 
  For a pseudo-Boolean function $f(x)$ on $\{0,1\}^n$, we say that $g(x,y)$ is a \emph{quadratization} of $f$
  if $g(x,y)$ is a quadratic polynomial depending on $x$ and on $m$ \emph{auxiliary} binary variables 
  $y_1,y_2,\ldots,y_m,$ such that
  \begin{equation*}\label{eq:quad}
    f(x)= \min \bigl\{ g(x,y) : y \in \{0,1\}^m \bigr\} \quad \mbox{for all } x \in \{0,1\}^n.
  \end{equation*}
\end{definition}

Clearly, if $g(x,y)$ is a quadratization  of $f$, then
\[
  \min \bigl\{ f(x) : x \in \{0,1\}^n \bigr\} = \min \bigl\{ g(x,y) : x \in \{0,1\}^n, y \in \{0,1\}^m \bigr\},
\]
so that the minimization of $f$ is reduced through this transformation to the QPBO problem of minimizing 
$g(x,y)$. We are also interested (see~\cite{ABCG}) in special types of quadratizations, which we call  
\emph{$y$-linear quadratizations}, which contain no products of auxiliary variables. If $g(x,y)$ is a 
$y$-linear quadratization, then $g$ can be written as 
\[
  g(x,y)=q(x) + \sum_{i=1}^m a_i(x) y_i,
\]
where $q(x)$ is quadratic in $x$ and each $a_i(x)$ is a linear function of $x$. When minimizing $g$ over 
$y$, each product $a_i(x)y_i$ takes the value $\min\{0,a_i(x)\}$. Thus, $y$-linear quadratizations can be 
viewed as piecewise linear functions of the $x$-variables.

\begin{example} 
  As an easy explicit example, consider the {\em negative monomial} 
  \[
    -\prod_{i=1}^n x_i=-x_1x_2\cdots x_n.
  \]
  This elementary pseudo-Boolean function has the following \emph{standard quadratization} (Freedman and 
  Drineas~\cite{FreedDrineas2005}): 
  \[
    s_n(x_1,x_2,\ldots, x_n,y) = y\left(n-1-\sum_{i=1}^n x_i\right).
  \]
  The reason is as follows: unless all the $x_i$ are $1$, then the quantity in parentheses in the expression 
  for $s_n$ is non-negative and the minimum value of $s_n$ is therefore $0$, obtained when $y=0$; and, if all 
  $x_i$ are $1$, the expression equals $-y$, minimized when $y=1$, giving value $-1$. In both cases, the minimum 
  value of $s_n$ is the same as the value of the negative monomial. 
\end{example}

\begin{example} 
  The {\em positive monomial} is the function
  \[
    \prod_{i=1}^n x_i=x_1x_2\cdots x_n.
  \]
  Ishikawa~\cite{Ishikawa2011} showed that it can be quadratized using $\left\lfloor \frac{n-1}{2}\right\rfloor$ 
  auxiliary variables, and this is currently the best available bound for positive monomials; see 
  also~\cite{FGBZ2011,Ishikawa2011}. 
\end{example}

Rosenberg~\cite{Rosenberg75} has proved that every pseudo-Boolean function $f(x)$ has a quadratization, 
and that a quadratization can be efficiently computed from the polynomial expression of $f$. This also easily 
follows from our foregoing observations that every monomial has a quadratization. (It is also the case --- 
see~\cite{ABCG} --- that any pseudo-Boolean function has a $y$-linear quadratization.)  Of course, quadratic 
PBO problems remain difficult in general, but this special class of problems has been thoroughly studied for 
the last few decades, and much progress has been made in solving large instances of QPBO, either exactly or 
heuristically. Quadratization has emerged in recent years as one of the most successful approach to the solution 
of very large-scale PBO problems arising in computer vision applications. (See, for instance, Boykov, Veksler 
and Zabih \cite{BVZ01}, Kolmogorov and Rother~\cite{KR2007}, Kolmogorov and Zabik~\cite{KZ2004}, Rother, 
Kolmogorov, Lempitsky and Szummer~\cite{RKLS2007}, Boros and Gruber~\cite{BG2011}, Fix, Gruber, Boros and 
Zabih~\cite{FGBZ2011}, Freedman and Drineas~\cite{FreedDrineas2005}, Ishikawa~\cite{Ishikawa2011}, Ramalingam, 
Russell, Ladick\'{y} and Torr~\cite{RRLT2011}, Rother, Kohli, Feng and Jia~\cite{RKFJ}.)

In a related paper, the present authors~\cite{ABCG} initiated a systematic study of quadratizations of 
pseudo-Boolean functions. We investigated the minimum number of auxiliary $y$-variables required in a 
quadratization of an arbitrary pseudo-Boolean function. In this paper, our focus is on {\em symmetric} 
pseudo-Boolean functions. A symmetric pseudo-Boolean function is one in which the value of the function 
depends only on the weight of the input. More precisely, a pseudo-Boolean function $f: \{0,1\}^n \to \bbr$ 
is {\em symmetric} if there is a function $k: \{0,1,\ldots, n\} \to \bbr$ such that $f(x)=k(l)=k_l$ where 
$l=|x|=\sum_{j=1}^n x_j$ is the Hamming weight (number of ones) of $x$. In another way, $f$ is symmetric 
if it is invariant under any permutation of the coordinates $\{1,2,\ldots,n\}$ of its variables. Note, 
for instance, that the positive and negative monomials are symmetric. Here, we investigate the number of 
auxiliary variables required in quadratizations of such functions. 

\subsection{Outline} 
In Section~\ref{sec:rep}, we present a representation theorem and corollaries, which provide useful ways 
of expressing symmetric pseudo-Boolean functions. In Section~\ref{sec:reptoquad} we explain how we can use 
such representations to construct quadratizations, and we present the implications for upper bounds on the 
number of auxiliary variables in Section~\ref{sec:numberaux}. Section~\ref{sec:lower} presents two types of 
lower bounds on the number of auxiliary variables for quadratizations of symmetric functions: an existence 
result, establishing that there are symmetric functions needing a rather large number of auxiliary variables; 
and a concrete lower bound on the number of auxiliary variables in any $y$-linear quadratization of the parity 
function. 

\section{A representation theorem}\label{sec:rep}
We introduce a useful piece of notation: for any real number $a$, $\n{a}$ denotes $\min(a,0)$, the smaller 
of $a$ and $0$. In this section, we present a result that will be key in our approach to obtaining quadratizations. 
This is a `representation theorem' that expresses a symmetric pseudo-Boolean function on variables $x_1, x_2, \ldots, x_n$ 
as a linear combination of terms of the form $\n{a-\sum_{r=1}^n x_r}$, for a suitable range of values $a$. 

Our main result, in its most general form, is as follows. 

\begin{theorem} \label{thm:repepsilongeneral} 
  Let  $0<\epsilon_i \leq 1$, for $i=0,1,\ldots,n$. Then every symmetric pseudo-Boolean function $f: \{0,1\}^n\to \bbr$ 
  can be represented uniquely in the form
  \begin{equation*}\label{eq:symrepresent}
    f(x)=\sum_{i=0}^{n}\alpha_i \n{i-\epsilon_i-\sum_{r=1}^n x_r}. 
  \end{equation*}
\end{theorem} 

When all the $\epsilon_i$ are equal, we can be more explicit about the coefficients in this representation.
Recall that $k_l=k(l)$ is the value of $f(x)$ in any point $x$ with Hamming weight equal to $l$. We set
$k_{-1}=0$ by convention.

\begin{theorem}\label{thm:repepsilon}
  Let $0<\epsilon \le 1$. Then every symmetric pseudo-Boolean function
  $f: \{0,1\}^n\to \bbr$ can be represented uniquely in the form
  \begin{equation*}\label{eq:symrepresent2}
    f(x)=\sum_{i=0}^{n}\alpha_i \n{i-\epsilon-\sum_{r=1}^n x_r}   
  \end{equation*}
  where, for $j=0,1,\ldots, n$, the value of $\alpha_j$ is
  \begin{equation}\label{eq:alpha}
    \alpha_j=- \sum_{i=0}^{j-2} \frac{(\epsilon-1)^{j-i-2}}{\epsilon^{j-i+1}}  k_i
             + \left(\frac{1}{\epsilon} +\frac{1}{\epsilon^2}\right) k_{j-1} -\frac{1}{\epsilon} k_j.
  \end{equation}
\end{theorem}
(The first sum in~\eqref{eq:alpha} is, by usual convention, taken to be $0$ if $j<2$.)

\begin{proof}[Proof of Theorem~\ref{thm:repepsilongeneral}]
  When $\sum_{r=1}^n x_r =j$, we should have $f(x)=k_j$. So, to find a coefficient vector 
  $\ve{\alpha}=(\alpha_0, \alpha_1, \ldots, \alpha_{n})^T \in \bbr^{n+1}$ which establishes the required 
  representation, we need to find a solution to the following system of $n+1$ linear equations: 
  \begin{equation}\label{eq:alpha2}
    k_j = \sum_{i=0}^{n}\alpha_i \n{i-\epsilon_i-j}
        = \sum_{i=0}^{j}\alpha_i \, (i-\epsilon_i-j) \quad \mbox{for } j=0,1,\ldots,n.
 \end{equation} 
 The matrix underlying this system is the lower-triangular matrix 
 \begin{align*} 
   A &= \left(\n{q -\epsilon_q -p}\right)_{p,q=1,2,\ldots,n+1}\\ 
     &= 
   \begin{pmatrix}
     -\epsilon_0 & 0 & 0 & 0 & \cdots & 0 \cr
     -1-\epsilon_0 & -\epsilon_1 & 0 &  0 & \cdots & 0 \cr
     -2-\epsilon_0 & -1-\epsilon_1 & -\epsilon_2 & 0 & \cdots & 0 \cr
     \vdots & \vdots & \vdots & \vdots & \ddots &  \vdots \cr
     -n-\epsilon_0 & -n+1-\epsilon_1 & -n+2-\epsilon_2 & -n+3-\epsilon_3 & \cdots & -\epsilon_n 
   \end{pmatrix}\\
  \end{align*}
  Because $A$ is lower-triangular with nonzero diagonal entries $-\epsilon_{q}$ ($q=0,1,\ldots,n$), this 
  system does indeed have a unique solution and therefore the representation exists and is unique. 
\end{proof}

\begin{proof}[Proof of Theorem~\ref{thm:repepsilon}]
  We check that a solution (and hence {\em the} solution) of the system~\eqref{eq:alpha2} with
  $\epsilon_i=\epsilon$ for all $i=0,1,\ldots,n$, is given by~\eqref{eq:alpha} in the statement of the theorem. 

  We proceed by induction on $j$. The case $j=0$ is easily verified, since the first equation
  in~\eqref{eq:alpha2} immediately yields $\alpha_0=-\frac{1}{\epsilon} k_0$, in agreement 
  with~\eqref{eq:alpha}. Assume now that (\ref{eq:alpha}) is satisfied by the values of $\alpha_i$ up 
  to $i=j-1$. Then, from (\ref{eq:alpha2}) and from the induction hypothesis, 
  \begin{align}
    -\epsilon \alpha_j &= k_j + \sum_{i=0}^{j-1}\alpha_i \, (\epsilon+j-i) \nonumber\\
    &= k_j + \sum_{i=0}^{j-1}\alpha_i \, (\epsilon+j-1-i) + \sum_{l=0}^{j-1}\alpha_l \nonumber\\
    &= k_j - k_{j-1} + \sum_{l=0}^{j-1}\alpha_l. \label{eq:tech1}
  \end{align}
  Substituting (\ref{eq:alpha}) in the last term of~(\ref{eq:tech1}) yields
  \begin{align}
    \sum_{l=0}^{j-1}\alpha_l &= \sum_{l=0}^{j-1} \left[
      - \sum_{i=0}^{l-2} \frac{(\epsilon-1)^{l-i-2}}{\epsilon^{l-i+1}}  k_i
      + \left(\frac{1}{\epsilon} +\frac{1}{\epsilon^2}\right) k_{l-1} -\frac{1}{\epsilon} k_l \right] \nonumber\\
    &= - \sum_{i=0}^{j-3} k_i  \sum_{l=i+2}^{j-1} \frac{(\epsilon-1)^{l-i-2}}{\epsilon^{l-i+1}}
      + \left(\frac{1}{\epsilon} +\frac{1}{\epsilon^2}\right) \sum_{l=0}^{j-1} k_{l-1} -\frac{1}{\epsilon} 
      \sum_{l=0}^{j-1} k_l \nonumber\\
    &= - \sum_{i=0}^{j-3} k_i \sum_{t=0}^{j-i-3} \frac{(\epsilon-1)^{t}}{\epsilon^{t+3}}
      - \frac{1}{\epsilon} k_{j-1} +  \frac{1}{\epsilon^2} \sum_{l=0}^{j-1} k_{l-1}\label{eq:penultimate}\\
    &= \sum_{i=0}^{j-3}    \frac{(\epsilon-1)^{j-i-2}}{\epsilon^{j-i}} \, k_i
      - \frac{1}{\epsilon} k_{j-1} +  \frac{1}{\epsilon^2}  k_{j-2} \label{eq:tech2}
  \end{align}
  where the last equality is obtained by summing the geometric series which appears in the first sum 
  of equation~(\ref{eq:penultimate}).

  Combining (\ref{eq:tech1}) and (\ref{eq:tech2}), we find
  \[
    \alpha_j = - \sum_{i=0}^{j-3}  \frac{(\epsilon-1)^{j-i-2}}{\epsilon^{j-i+1}} \, k_i
    -\frac{1}{\epsilon^3} k_{j-2} + \left(\frac{1}{\epsilon} +\frac{1}{\epsilon^2}\right) k_{j-1} -\frac{1}{\epsilon} k_j,
  \]
  which is equivalent to (\ref{eq:alpha}).
\end{proof}

There are two special cases of Theorem~\ref{thm:repepsilon} that we will use in particular. 
When $\epsilon=1/2$, Theorem~\ref{thm:repepsilon} yields:

\begin{corollary}\label{thm:rep2} 
  Every symmetric pseudo-Boolean function $f: \{0,1\}^n\to \bbr$ can be represented uniquely in the form
  \[
    f(x)=\sum_{i=0}^{n}\alpha_i \n{i-\frac12-\sum_{r=1}^n x_r}
  \]
  where
  \[
    \alpha_i=-8 \sum_{j=0}^i (-1)^{i-j}k_j - 2k_{i-1}+6 k_i
  \]
  for $i=0,1,\ldots,n$, and  $k_{-1}=0$.
  \qed
\end{corollary}

Taking $\epsilon=1$ in Theorem~\ref{thm:repepsilon}, we obtain the following. 
 
\begin{corollary}\label{thm:fix} 
  Every symmetric pseudo-Boolean function $f: \{0,1\}^n\to \bbr$ can be represented in the form
  \[
    f(x)=k_0+(k_1-k_0)\sum_{r=1}^n x_r +\sum_{i=1}^{n-1}\left(-k_{i-1}+2k_i-k_{i+1}\right) \n{i-\sum_{r=1}^n x_r}.
  \]
  \qed
\end{corollary} 

In fact, Corollary~\ref{thm:fix} follows from work of Fix~\cite{Fix}, and a simpler direct proof can be given. 
As Fix observed, if  $\sum_{r=1}^n x_r=l$, then 
\[
  f(x)=k(l)=\sum_{i=0}^n k(i)\delta_i(l)
\]
where $\delta_i(l)=1$ if $i=l$ and $\delta_i(l)=0$ otherwise. Then, it can be seen that 
\[
  \delta_i(l)=-\n{i-1-l}+2\n{i-l}-\n{i+1-l}.
\]
From this, it follows that
\[
  f(x)=\sum_{i=0}^n k(i)\left(-\n{i-1-l}+2\n{i-l}-\n{i+1-l}\right).
\] 
On simplification, this gives
\[
  f(x)=k_0+l(k_1-k_0) + \sum_{i=1}^{n-1} (-k_{i-1}+2k_i-k_{i+1})\n{i-l},
\] 
as required. 

\section{From representations to quadratizations}\label{sec:reptoquad}
In this section we explain how a representation of the type presented in the previous section can be 
used to construct  quadratizations of pseudo-Boolean functions. One useful observation is that when a 
coefficient $\alpha_i$ is non-negative, the corresponding term $\alpha_i \n{i-\epsilon_i-\sum_{r=1}^n x_r}$ 
in the representation of Theorem~\ref{thm:repepsilongeneral} of $f$ can be quadratized as  
$\min_{y_i} \alpha_i y_i (i-\epsilon_i-\sum_{r=1}^n x_r)$. But this translation simply does not work if 
$\alpha_i$ is negative. The strategy described in this section is to take an expression as given in 
Theorem~\ref{thm:repepsilongeneral} (or one of its special cases) and add a quantity that is identically-$0$ 
and which will result in a final expression that has no terms with negative coefficients. The following 
Lemma describes three possible such quantities. The first is going to be useful when working with 
representations of the form given in Corollary~\ref{thm:fix}, and the second and third will be useful when 
working with the representations from Corollary~\ref{thm:rep2}.  

\begin{lemma}\label{lem:zero1}
  Let
  \begin{align*} 
    E(l) &=l(l-1)+2\sum_{i=1}^{n-1} \n{i-l}, \\
    E'(l)&=\frac{l(l-1)}{2}+ 2 \sum_{\stackrel{i=2:}{i \; \text{even}}}^{n}\n{i-\frac12-l}, \label{eq:zero1}
  \end{align*}
  and
  \begin{equation*}\label{eq:zero2}
    E''(l)=\frac{l(l+1)}{2}+ 2 \sum_{\stackrel{i=1:}{i \; \text{odd}}}^{n}\n{i-\frac12-l}.
  \end{equation*}
  Then, for all $l=0,1,\ldots,n$, $E(l)= E'(l)=E''(l)=0$.
\end{lemma}

\begin{proof}
  First we show that $E(l)$ is identically-$0$. We have 
  \begin{align*} 
    E(l) &= l(l-1)+2\sum_{i=1}^{n-1} \n{i-l} \\ 
         &= l(l-1)+2\sum_{i=1}^{l-1} (i-l) \\
         &= l(l-1)-2\sum_{j=1}^{l-1} j \\
         &= l(l-1)-l(l-1)=0.
  \end{align*} 
  We next show that $E'(l) = 0$ for all values of $l$. Fix $l$ and
  note that $i-\frac12-l \leq 0$ if and only if $i\leq l$. Hence,
  \begin{equation}\label{eqn:E}
      \sum_{\stackrel{i=2:}{i \; \text{even}}}^{n} \n{i-\frac12-l}
    = \sum_{\stackrel{i=2:}{i \; \text{even}}}^{l} \left(i-\frac12-l\right)
    = \sum_{\stackrel{i=2:}{i \; \text{even}}}^{l}  i - \left(\frac12+l\right) \left\lfloor \frac{l}{2}\right\rfloor.
  \end{equation}
  By considering separately the cases where $l$ is respectively even or odd, one can conclude
  that $E'(l)=0$ for all $l =0,\ldots,n$. For, if $l=2r$, then the expression on the right in 
  equation~(\ref{eqn:E}) is $r/2-r^2=-l(l-1)/4$ and, if $l=2r+1$, it is $-r/2-r^2=-l(l-1)/4$.
  The identity $E''(l) =0$ (for all $l$) can be proved similarly, or can be deduced from the 
  previous one by observing  that, for all $l=0,1,\ldots,n$,
  \[
    \sum_{i=1}^{n} \n{i-\frac12-l} = \sum_{i=1}^{l} \left(i-\frac12-l\right) = -\frac12 l^2.
  \]
  We then can note that $$E'(l)+E''(l)=l^2+2\sum_{i=1}^n \n{i-\frac{1}{2}-l}=l^2-l^2=0,$$ so that $E''=-E'=0$.
\end{proof}

We gave a direct, self-contained, proof of Lemma~\ref{lem:zero1}, but in fact these three identities follow 
from Corollary~\ref{thm:rep2} and Corollary~\ref{thm:fix}. For, if we apply Corollary~\ref{thm:fix} to the 
function $f(x)=\sum_{r=1}^n x_r \left(\sum_{r=1}^nx_r -1\right)$, we see that 
\[
  f(x)=-2\sum_{i=1}^{n-1}\n{i-\sum_{r=1}^n x_i},
\]
which implies the first identity of Lemma~\ref{lem:zero1}. Applying Corollary~\ref{thm:rep2} to $f(x)$ shows 
(after some calculation) that 
\[
  f(x)=-4\sum_{\stackrel{i=2:}{i \; \text{even}}}^{n}\n{i-\frac12-l},
\]
giving the second identity (that $E'$ is identically-$0$). Applying Corollary~\ref{thm:rep2} to the function 
$g(x)=\sum_{r=1}^n x_r \left(\sum_{r=1}^nx_r+1\right)$ yields the third identity. 

\section{Upper bounds on number of auxiliary variables}\label{sec:numberaux}
\subsection{Any symmetric function} 
We first have the following very general result, which provides an explicit construction of a quadratization 
of any pseudo-Boolean function, using no more than $n-2$ auxiliary variables. 

\begin{theorem}\label{thm:symm_quad}
  Every symmetric function of $n$ variables can be quadratized using $n-2$ auxiliary variables.
\end{theorem}

\begin{proof}
  Using Corollary~\ref{thm:rep2}, we can write any symmetric function $f$ as
  \[
    f(x)= -\alpha_0 \left(\frac12+\sum_{j=1}^n x_j\right) + \sum_{i=1}^{n}\alpha_i \n{i-\frac12-\sum_{j=1}^n x_j}. 
  \]
  Let $\alpha_r = \min\{\alpha_i: i \mbox{ even}, i \geq2\}$ and $\alpha_s= \min\{\alpha_i: i \mbox{ odd}\}$.
  Now add to $f$ the expression
  \[
    -\frac{{\alpha_r}}{2}E'\left(\sum_{j=1}^n x_j\right) -\frac{{\alpha_s}}{2}E''\left(\sum_{j=1}^n x_j\right),
  \] 
  which is identically-$0$. This results in an expression for $f$ of the form
  \[
    f(x)= a_0 + a_1\sum_{j=1}^n x_j +a_2 \sum_{1\leq i<j \leq n} x_ix_j
    + \sum_{i=1}^{n}\beta_i \n{i-\frac12-\sum_{j=1}^n x_j},
  \]
  where, for each $i$, if $i$ is even, $\beta_i=\alpha_i-{\alpha_r} \ge 0$, and if $i$ is odd, 
  $\beta_i=\alpha_i-{\alpha_s} \ge 0$. So all the coefficients $\beta_i$ are non-negative. Furthermore, 
  $\beta_r=\beta_s=0$, so we have an expression for $f$ involving no more than $n-2$   positive coefficients 
  $\beta_i$. Then,
  \[
    g(x,y)= a_0 + a_1\sum_{j=1}^n x_j +a_2 \sum_{1\leq i<j \leq n} x_ix_j
    + \sum_{\stackrel{i=1:}{i \neq r,s}}^{n}\beta_i y_i\left(i-\frac12-\sum_{j=1}^n x_j\right)
  \]
  is a quadratization of $f$ involving at most $n-2$ auxiliary variables.
\end{proof}
(A construction in~\cite{Fix} shows an upper bound of $n-1$. This is obtained by adding a multiple of 
$E(\sum_{r=1}^n x_r)$ to each term in the expression from Corollary~\ref{thm:fix}, rather than to the 
expression as a whole, resulting in more complex quadratizations.)  
 
Notice that the quadratization in the proof of Theorem~\ref{thm:symm_quad} is $y$-linear, so we have in fact shown: 

\begin{theorem}\label{thm:upperlinear} 
  Every symmetric function of $n$ variables has a $y$-linear quadratization involving at most 
  $n-2$ auxiliary variables.
\end{theorem}

Furthermore, these quadratizations are also symmetric in the $x$-variables. Not every quadratization of a 
symmetric function must itself be symmetric in the original variables. For example, consider the negative 
monomial 
\[
  -\prod_{i=1}^n x_i=-x_1x_2\cdots x_n.
\] 
As we have seen, this has the quadratization $y\left(n-1-\sum_{j=1}^n x_j\right)$, which is symmetric. 
However, it also has the quadratization 
\[
  (n-2)x_{n}y - \sum_{i=1}^{n-1} x_i (y-\xbar_n),
\] 
where $\xbar_n=1-x_n$, which is not symmetric in the $x$-variables. 

\subsection{Monomials} 
The quadratization of monomials (positive and negative) has been fairly well-studied. The \emph{standard quadratization} 
of the negative monomial 
\[
  f(x)=-\prod_{i=1}^n x_i=-x_1x_2\cdots x_n,
\]
is 
\[
  s_n(x_1,x_2,\ldots, x_n,y)=y\left(n-1-\sum_{j=1}^n x_j\right).
\]
(A related paper by the present authors~\cite{ABCG} gives a complete characterization of all the quadratizations 
of negative monomials involving one auxiliary variable and this is, in a sense, one of  the simplest.) If we apply 
Corollary~\ref{thm:rep2} to the negative monomial, noting that $k_i=0$ for $i<n$ and $k_n=-1$, we obtain the 
representation 
\[
  f(x)=2\n{n-\frac{1}{2}-\sum_{r=1}^nx_r},
\]
which immediately leads to the quadratization 
\[
  h=2y\left(n-\frac{1}{2}-\sum_{r=1}^nx_r\right),
\]
only slightly different from the standard one. We could, instead, apply Corollary~\ref{thm:fix}, which would show 
that $f(x)=\n{n-1-\sum_{r=1}^nx_r},$ from which we immediately obtain the standard quadratization. 

As we noted earlier, the best known result (smallest number of auxiliary variables) for positive monomials is that 
they can be quadratized using $\left\lfloor \frac{n-1}{2} \right\rfloor$ auxiliary variables. This was shown by  
Ishikawa~\cite{Ishikawa2011}. We can see that this many auxiliary variables suffice by using our representation 
theorem, Corollary~\ref{thm:rep2}, together with the argument given in the proof of Theorem~\ref{thm:symm_quad}. 

\begin{theorem} \label{thm:positiveterm}
  The positive monomial $P=\prod_{i=1}^n x_i$ can be quadratized using $\left\lfloor \frac{n-1}{2} \right\rfloor$ 
  auxiliary variables.
\end{theorem}

\begin{proof}
  Consider first the case where $n$ is even. By Corollary~\ref{thm:rep2}, noting that $k_i=0$ for $i<n$ and $k_n=1$, 
  we have $P = -2 \n{n-\frac{1}{2}-l}$ where $l=\sum_{r=1}^n x_r$. By Lemma~\ref{lem:zero1},
  \begin{align*}
    P &= -2 \n{n-\frac{1}{2}-l} + E'(l) \\
      &= \frac{l(l-1)}{2}+ \sum_{\stackrel{i=2:}{i \; \text{even}}}^{n-2}2\n{i-\frac{1}{2}-l} \\
      &= \sum_{1\leq i<j \leq n} x_ix_j + \min_y \sum_{\stackrel{i=2:}{i \; \text{even}}}^{n-2}2 y_i\left(i-\frac{1}{2}-l\right).
  \end{align*}
  This provides the required quadratization using $\frac{n}{2}-1=\left\lfloor \frac{n-1}{2} \right\rfloor$ new 
  variables.

  When $n$ is odd, one similarly derives the following from Lemma~\ref{lem:zero1}:
  \begin{align*}
    P &= -2 \n{n-\frac{1}{2}-l} + E''(l) \\
      &= \sum_{i=1}^n x_i + \sum_{1\leq i<j \leq n} x_ix_j 
       + \min_y \sum_{\stackrel{i=1:}{i \; \text{odd}}}^{n-2}2 y_i\left(i-\frac{1}{2}-l\right).
  \end{align*}
\end{proof}

This quadratization of $P$ requires the same number of auxiliary variables as Ishikawa's construction. Both 
quadratizations are, in fact, identical when $n$ is even, but appear to be different when $n$ is odd.

Note that an alternative approach to the case of odd $n$ would be as follows. Write
\[
  P = \prod_{i=1}^{n-1} x_i - \prod_{i=1}^{n-1} x_i \xbar_n,
\]
where $\xbar_n=1-x_n$. The first term can now be quadratized using $\frac{n-1}{2}-1$ new variables (since it 
contains an even number of variables), and the second term, viewed as a negative monomial in $x_1,\ldots,x_{n-1},\xbar_n$, 
has a standard quadratization requiring one further auxiliary variable. Thus, this leads again
to a quadratization of $P$ with $\frac{n-1}{2}=\left\lfloor \frac{n-1}{2}\right\rfloor$ new variables.
This quadratization is also different from Ishikawa's. 

\subsection{$t$-out of $n$ and exact-$t$ functions} 
Consider now the \emph{$t$-out-of-$n$} function defined by: 
\[
  f_{t,n}(x)=1\qquad \text{if and only if}\qquad \sum_{i=1}^n x_i \geq t.
\]

The basic Boolean functions $\textrm{And}_n(x):=\prod_{i=1}^n x_i$ (a positive monomial) and 
$\textrm{Or}_n(x):=1-\prod_{i=1}^n(1-x_i)$ are examples of $t$-out of $n$ functions with $t=n$
and $t=1$, respectively. Another popular example is the \emph{majority} function given by:
\[
  \textrm{Maj}_n(x) := \begin{cases}
    1 & \ \text{if}\ \sum_{i=1}^n x_i\geq \left\lceil n/2 \right\rceil,\\
    0 & \ \text{otherwise},
  \end{cases}
\]
which breaks ties in favor of ones when $n$ is even. In this case, $t=\left\lceil n/2 \right\rceil$.

\begin{corollary}\label{thm:k-out-of-n}
  The $t$-out-of-$n$ function $f_{t,n}$ can be quadratized using $\left\lceil n/2 \right\rceil$ 
  auxiliary variables.
\end{corollary}

\begin{proof}
  From Corollary~\ref{thm:rep2}, $f_{t,n}$ can be represented in the form 
  \begin{equation}\label{eq:ftn}
    f_{t,n}(x)=\sum_{i=0}^{n}\alpha_i \n{i-\frac12-\sum_{j=1}^n x_j}
  \end{equation}
  where $\alpha_i=0$ when $i<t$, $\alpha_t=-2$, and $\alpha_i=4(-1)^{i-t-1}$ when $i>t$.

  Since the terms of $f_{t,n}$ alternate in sign when $i\geq t$, we can again use Lemma~\ref{lem:zero1} 
  to make all coefficients non-negative by adding either $2E'(l)$ or $2E''(l)$ to~(\ref{eq:ftn}), depending 
  on the parity of $t$. The resulting expression has $\left\lceil n/2 \right\rceil$ positive 
  coefficients, and its remaining coefficients are zero. Thus, it can be quadratized with
  $\left\lceil n/2 \right\rceil$ auxiliary variables.
\end{proof}

A related function is the \emph{exact-$t$} (out of $n$) function, defined as $f^=_{t,n}(x)=1$ if and only if
the Hamming weight of $x$ equals $t$. Using Corollary~\ref{thm:rep2} again, we have that $f^=_{t,n}$ 
can be represented in the form given in \eqref{eq:ftn} with $\alpha_i=0$ when $i<t$, $\alpha_t=-2$,
$\alpha_{t+1}=6$, and $\alpha_i=0$ when $i>t+1$. Depending on the parity of $t$, we add $E'(l)$ or
$E''(l)$ to \eqref{eq:ftn} to obtain an expression with $\left\lfloor n/2\right\rfloor$ positive
coefficients, which can then be quadratized with $\left\lfloor n/2\right\rfloor$ auxiliary variables.
We have just proved the following:

\begin{corollary}\label{thm:exact-k}
  The exact-$t$ function $f^=_{t,n}$ can be quadratized using $\left\lfloor n/2 \right\rfloor$ 
  auxiliary variables.
  \qed
\end{corollary}

The positive monomial and the $\textrm{And}_n$ Boolean function are also special cases of exact-$t$ 
functions, both with $t=n$. It is apparent from the argument leading to Corollary~\ref{thm:exact-k}
that the reason the positive monomial (and hence, the $\textrm{And}_n$ function) requires 
$\left\lfloor \frac{n-1}{2} \right\rfloor$ auxiliary variables instead of $\left\lfloor n/2 \right\rfloor$
is precisely because $t=n$.

\subsection{Parity and its complement} 
The \emph{parity} function is the (pseudo-)Boolean function $\Pi(x)$ such that $\Pi(x)=1$ if the Hamming weight 
of $x$ is odd, and $\Pi(x)=0$ otherwise. To derive a quadratization of this function, we will use 
Corollary~\ref{thm:fix} rather than Corollary~\ref{thm:rep2}, and will make use of a variant of the argument 
given to establish Theorem~\ref{thm:symm_quad}. By Corollary~\ref{thm:fix}, we can see that $\Pi$ has the 
representation 
\begin{equation}\label{parity} 
  \Pi(x) = \sum_{j=1}^n x_j+2\sum_{i=1}^{n-1}(-1)^{i-1}\n{i-\sum_{j=1}^n x_j}.
\end{equation} 

Let $E(l)$ be as in Lemma~\ref{lem:zero1}. By adding $E(\sum_{j=1}^n x_j)$ to this representation of $\Pi$, 
we obtain a representation with non-negative coefficients, which leads to a quadratization with 
$m=\left\lfloor n/2\right\rfloor$ auxiliary variables: $\Pi(x)=\min_{y \in \{0,1\}^m} g(x,y)$ where
\[
  g(x,y)=2 \sum_{i<j}x_ix_j+ \sum_{j=1}^n x_j + 4\,\sum_{\stackrel{i=1:}{i \;\text{odd}}}^{n-1} y_i \left(i-\sum_{j=1}^n x_j\right).
\]
(The terms with coefficient $-2$ in the expansion (\ref{parity}) disappear on the addition of $E$.)

The complement, $\overline{\Pi}$ of $\Pi$ can be represented as 
\[
  \overline{\Pi}(x)= 1-\sum_{j=1}^n x_j +2\sum_{i=1}^{n-1} (-1)^{i} \n{i -\sum_{j=1}^n x_j},
\]
so, by adding $E(\sum_{j=1}^n x_j)$, to eliminate negative coefficients, we arrive at the following 
quadratization involving $m=\left\lfloor \frac{n-1}{2}\right\rfloor$ auxiliary variables: 
\[
  g'(x,y)=1+2 \sum_{i<j}x_ix_j -\sum_{j=1}^n x_j +4\,\sum_{\stackrel{i=2:}{i \;\text{even}}}^{n-1} y_i \left(i-\sum_{j=1}^n x_j\right).
\]
So we conclude the following: 

\begin{theorem} 
  The parity function of $n$ variables has a $y$-linear quadratization involving $\left\lfloor n/2\right\rfloor$ 
  auxiliary variables, and its complement has a $y$-linear quadratization involving $\left\lfloor \frac{n-1}{2}\right\rfloor$ 
  auxiliary variables.
\end{theorem}

\section{Lower bounds on the number of auxiliary variables}\label{sec:lower}
\subsection{Generic lower bounds}
The following result is inspired by (but is different and does not follow from) a transformation given in Siu, 
Roychowdhury and Kailath~\cite{Siu95}, in the framework of the representation of Boolean functions by threshold 
circuits. This result relates quadratizations of arbitrary (possibly non-symmetric) pseudo-Boolean functions to 
the quadratization of symmetric functions on a larger, related, number of variables. We will then use a lower 
bound result from~\cite{ABCG} in order to obtain a lower bound result for symmetric functions. 

\begin{lemma}~\label{lem:arbtosym}
  Suppose that $n,m $ are positive integers and suppose that every symmetric pseudo-Boolean function 
  $F(z)$ of $N=2^n-1$ variables (that is, every symmetric function $F:\{0,1\}^{2^n-1}\to \bbr$) has an 
  $m$-quadratization. Then every  (arbitrary) pseudo-Boolean function $f(x)$ on $\{0,1\}^n$ also has an 
  $m$-quadratization.
\end{lemma}

\begin{proof}
  Let $f(x)$ be an arbitrary pseudo-Boolean function of $n$ variables. We are going to construct a sequence of 
  four functions $k$, $F$, $G$, $g$, such that $g$ is a quadratization of $f$. For this purpose, let $N=2^{n}-1$.
  \begin{enumerate}
  \item 
    Let $k:\{0,1,\ldots, N\} \to \bbr$ be defined as follows: $k(w) := f(x)$ where $x$ is the binary representation 
    of $w$, that is, $w=\sum_{i=1}^n 2^{i-1} x_i$.

  \item 
    Let $F$ be the symmetric pseudo-Boolean function of $N$ variables defined by: for all $z \in \{0,1\}^N$, 
    $F(z) := k(|z|)$, where $|z|$ is the Hamming weight of $z$. (This defines $F$ completely, given that it is 
    symmetric.) 

  \item 
    Let $G(z,y)$ be an arbitrary quadratization of $F(z)$ using $m$ auxiliary variables. (The hypothesis of 
    the theorem is that such quadratizations exist.) 

  \item 
    Finally, let $g(x,y)$ be the pseudo-Boolean function on $\{0,1\}^{n+m}$ that is obtained by identifying each of the 
    variables $z_{2^{j-1}}, z_{2^{j-1}+1},\ldots,z_{2^{j}-1}$ with $x_{j}$ in $G(z,y)$, for $j=1,2,\ldots,n$; that is,
    \[
      g(x_1,x_2,x_3,\ldots,x_n,y) := G(x_1,x_2,x_2,x_3,x_3,x_3,x_3,\ldots,x_n,\ldots,x_n,y).
    \]
    (The unification makes sense since $2^{j-1}x_j = z_{2^{j-1}} + z_{2^{j-1}+1} +\cdots + z_{2^{j}-1}$, for all 
    $j=1,2,\ldots,n$.) 
  \end{enumerate}
  We claim that $g(x,y)$ is a quadratization of $f$. Indeed, $g$ is clearly quadratic, because $G$ is. Moreover, 
  for every point $x \in \{0,1\}^n$,
  \begin{align}
    &\hspace*{-1truecm}\min \bigl\{ g(x,y) : y \in \{0,1\}^m \bigr\} \nonumber\\
    &=  \min \bigl\{ G(x_1,x_2,x_2,x_3,x_3,x_3,x_3,\ldots,x_n,y) : y \in \{0,1\}^m \bigr\} \label{eq:g1}\\
    &=  F(x_1,x_2,x_2,x_3,x_3,x_3,x_3,\ldots,x_n) \label{eq:g2}\\
    &=  k\left(\sum_{i=1}^n 2^{i-1} x_i\right) \label{eq:g3}\\
    &=  f(x) \label{eq:g4}
  \end{align}
  (where equality \eqref{eq:g1} is by definition of $g$, \eqref{eq:g2} is by definition of $G$, \eqref{eq:g3} is 
  by definition of $F$, and \eqref{eq:g4} is by definition of $k$).
\end{proof}

We will now make use of the following result from~\cite{ABCG}. 

\begin{theorem}\label{thm:lowergeneral} 
  There are pseudo-Boolean functions of $n$ variables for which any quadratization must involve at least 
  $\Omega(2^{n/2})$ auxiliary variables. 
\end{theorem} 

To be more concrete, the analysis in~\cite{ABCG} implies that  for any $n \ge 8$, there is a pseudo-Boolean 
function on $n$ variables for which any quadratization will require at least $2^{n/2}/8$ auxiliary variables. 

This leads to the following lower bound result for symmetric functions. 

\begin{theorem}\label{thm:lb_sym1}
  There exist symmetric functions of $n$ variables for which any quadratization must involve
  at least $\Omega(\sqrt{n})$ auxiliary variables.
\end{theorem}

\begin{proof}
  Lemma~\ref{lem:arbtosym} shows that, if every symmetric function $F(z)$ on $\{0,1\}^N$, with $N=2^n-1$, has 
  an $m$-quadratization, then every (arbitrary) function $f(x)$ on $\{0,1\}^n$ also has an $m$-quadratization. 
  On the other hand, from Theorem~\ref{thm:lowergeneral}, we know that some pseudo-Boolean functions on $n$ 
  variables require $\Omega(2^{n/2})$ auxiliary variables. It follows that some symmetric functions on $N$ 
  variables must need $\Omega(\sqrt{N})$ auxiliary variables in every quadratization.
\end{proof}

We also have a similar lower bound result for $y$-linear quadratizations. It rests on the following result 
from~\cite{ABCG}: 

\begin{theorem}\label{thm:lowerbound2}
  There are pseudo-Boolean functions of $n$ variables for which any $y$-linear quadratization 
  must involve at least $\Omega(2^n/n)$ auxiliary variables.
\end{theorem}

We then have the following. 

\begin{theorem}\label{thm:lb_sym2}
  There exist symmetric functions of $n$ variables for which any $y$-linear quadratization must involve
  at least $\Omega(n/\log n)$ auxiliary variables.
\end{theorem}

\begin{proof}
  The proof is similar to the previous one: it suffices to observe that when $G(z,y)$ is $y$-linear, then so 
  is $g(x,y)$, and to rely on the generic lower bound $\Omega(2^n/n)$ of Theorem~\ref{thm:lowerbound2} for the 
  number of auxiliary variables required in every $y$-linear quadratization of certain pseudo-Boolean functions. 
\end{proof}

Note that the lower bound in Theorem~\ref{thm:lb_sym2} for the number of auxiliary variables in $y$-linear 
quadratizations comes within a factor $O(\log n)$ of the upper bound of $n-2$ from Theorem~\ref{thm:upperlinear}. 

\subsection{A lower bound for the parity function}
The results just obtained prove the existence of symmetric pseudo-Boolean functions which require a significant 
number of auxiliary variables to quadratize. Specifically, there exist functions needing $\Omega(\sqrt{n})$ auxiliary 
variables in any quadratization, and functions needing $\Omega(n/\log n)$ auxiliary variables in any $y$-linear 
quadratization. Those results do not, however, explicitly exhibit particular such functions. We next give a concrete 
example of a function which needs a significant number of auxiliary variables in any $y$-linear quadratization. 

\begin{theorem}\label{thm:lb_parity}
  Every $y$-linear quadratization of the parity function on $n$ variables must involve
  at least $\Omega(\sqrt{n})$ auxiliary variables.
\end{theorem}

\begin{proof} 
  Let $g(x,y)$ be an arbitrary $y$-linear quadratization of the parity function. Then it can be written as 
  \begin{equation}\label{g} 
    g(x,y) = q(x) + \sum_{i=1}^m y_i (\ell_i(x) - b_i)
  \end{equation} 
  where $q(x)$ is quadratic, and $\ell_1(x),\ldots,\ell_m(x)$ are linear functions of $x$ only.

  For each $i\in [m]=\{1,2,\ldots,m\}$, consider the regions
  \[
    R^+_i= \{x \in \bbr^n : \ell_i(x) \geq b_i \}, \;\;R^-_i= \{x \in \bbr^n : \ell_i(x) \leq b_i \},
  \]
  which are closed half-spaces defined by the linear functions $\ell_i$. For each $S \subseteq [m]$, let $R_S$ 
  denote the region $R_S= \bigl(\bigcap_{i \in S} R^-_i \bigr) \cap \bigl( \bigcap_{i \not\in S} R^+_i \bigr)$. 
  This is one of the `cells' into which the $m$ hyperplanes defining the linear functions $\ell_i$ partition $\bbr^n$. 

  On every cell $R_S$, the function $f(x)=\min \{ g(x,y) : y \in \{0,1\}^m \}$ is quadratic. Indeed, on $R(S)$, 
  we have
  \[
    \min \bigl\{ g(x,y) : y \in \{0,1\}^m \bigr\} = q(x) + \sum_{i \in S} (\ell_i(x) - b_i).
  \]
  We now use a result from Saks~\cite{Saks93} and Impagliazzo, Paturi and Saks~\cite{IPS} (which was used to 
  obtain lower bounds on the size of threshold circuits representing the parity function). Let us say that a 
  set of hyperplanes {\em slices} all $r$-dimensional subcubes of the Boolean hypercube $\{0,1\}^n$ if for each 
  subcube (or face) of $\{0,1\}^n$ of dimension $r$, there are two vertices of the subcube that lie on opposite 
  sides of one of these hyperplanes. Then (Proposition~3.82 of~\cite{Saks93}), if a set of $m$ hyperplanes slices 
  all $r$-dimensional subcubes, we have $m>\sqrt{n/(r+1)-1}$. In particular, therefore, any set of hyperplanes 
  that slices every $3$-dimensional subcube of $\{0,1\}^n$ must contain more than $\sqrt{n/4 -1}$ planes. 
  Suppose the hyperplanes defined by the linear functions $\ell_i$ do {\em not} slice all $3$-dimensional subcubes. 
  Then there would be some cell $R_S$ containing a subcube of dimension $3$. The parity function restricted to that 
  subcube would then be equal to the quadratic expression $q(x)+\sum_{i \in S}(\ell_i(x)-b_i)$. However, it is 
  well-known (see, for instance~\cite{Saks93,MinskyPapert,WangWilliams}) that the parity function on a subcube 
  of dimension $r$ cannot be represented as a pseudo-Boolean function of degree less than $r$ (and it cannot even 
  be represented as the sign of a pseudo-Boolean function of degree less than $r$). So, we would then have a 
  quadratic, degree-$2$, representation of parity on a cube of dimension $3$, which is not possible. It follows, 
  therefore, that the set of hyperplanes in question must slice all $3$-dimensional subcubes and therefore has 
  size $m >\sqrt{n/4 -1}$.
\end{proof}

\section{Conclusions} 
In this paper, we have studied the number of auxiliary variables required in quadratizations (and $y$-linear 
quadratizations) of symmetric pseudo-Boolean function. We have presented explicit general constructions of 
quadratizations, via special types of representations of the functions. This shows that every such function 
can be quadratized (with a $y$-linear quadratization, symmetric in the original variables) using at most $n-2$ 
auxiliary variables. We investigated in more detail the quadratizations of special functions (monomials, 
$t$-out-of-$n$, exact-$t$, and parity functions), where it was possible to obtain quadratizations using significantly 
fewer than $n-2$ auxiliary variables. By drawing on a general result from our related paper~\cite{ABCG} and 
establishing a connection between quadratizations of general functions and of symmetric functions on a related 
number of variables, we showed that there exist symmetric functions requiring $\Omega(\sqrt{n})$ auxiliary 
variables in any quadratization, and that $y$-linear quadratization can require $\Omega(n/\log n)$ variables. 
It would clearly be of interest to close the gaps between these lower bounds and the linear upper bound. We 
established, further, that any $y$-linear representation of the parity function needs $\Omega(\sqrt{n})$ 
auxiliary variables. An open question is to determine whether a similar (or better) lower bound can be 
obtained for any (not necessarily $y$-linear) quadratization of this, or another specific, symmetric function.
For instance, any example of a symmetric function where a non $y$-linear quadratization needs fewer variables 
than the $y$-linear ones would be of interest, as would be any non constant lower bound on the number of
auxiliary variables for positive monomials. Furthermore, the number of positive quadratic terms in any known 
quadratization of the positive monomial is at least $n-1$, but no lower bound on such quantity has been found 
so far. Settling this question is also of great interest, as it is related to the quality of relaxations based 
on quadratizations for PBO problems.

\vskip 0.5cm
{\bf Acknowledgements.}
We thank Gy\"{o}rgy Tur\'an for several discussions and references on Boolean circuits for 
symmetric functions.
The second author thanks the National Science Foundation (Grant IIS-1161476) for partial support. 
The third author was partially funded by the Interuniversity Attraction Poles Programme initiated by 
the Belgian Science Policy Office (grant P7/36) and by a sabbatical grant from FNRS. 
The fourth author thanks the joint CAPES (Brazil)/Fulbright (USA) fellowship process BEX-2387050/15061676 
for partial suport.

\end{document}